\documentclass[11pt]{amsart}
\calclayout

\usepackage[pagebackref = false]{hyperref}
\usepackage{amsmath,xcolor}

\newtheorem{theorem}{Theorem}[section]
\newtheorem{lemma}[theorem]{Lemma}

\theoremstyle{definition}
\newtheorem{definition}[theorem]{Definition}

\theoremstyle{remark}

\newtheorem*{remark*}{Remark}

\theoremstyle{plain}

\newtheorem{corollary}[theorem]{Corollary}
\newtheorem{proposition}[theorem]{Proposition}
\newtheorem*{proposition*}{Proposition}

\theoremstyle{definition}

\newtheorem{observation}[theorem]{Observation}
\newtheorem*{observation*}{Observation}

\numberwithin{equation}{section}
\newcommand{\pa}{\mathfrak{a}}
\newcommand{\PA}{\mathfrak{A}}
\newcommand{\pb}{\mathfrak{b}}
\newcommand{\ve}{\varepsilon}

\newcommand{\uqf}[1]{\langle #1 \rangle} % <blabla>, i.e. unary quadratic forms
\newcommand{\FoF}[1]{\mathrm{Quot}(#1)} %Field of fractions = Quotient field

\let\phi\varphi

\newcommand{\vect}[1]{\mathbf{#1}}
\def\N{\mathbb N}
\def\Z{\mathbb Z}
\def\R{\mathbb R}
\def\Q{\mathbb Q}
\def\C{\mathbb C}
\def\O{\mathcal O}%Order
\def\P{\mathcal P}%Pythagoras number

\begin{document}
	
	\title{On quadratic Waring's problem in totally real number fields}
	
	
	%    author one information
	\author{Jakub Kr\'asensk\'y}
	\address{Charles University, Faculty of Mathematics and Physics, Department of Algebra, Sokolovsk\'{a}~83, 18600 Praha 8, Czech Republic}
	\curraddr{}
	\email{krasensky@karlin.mff.cuni.cz}
	\thanks{J.K. acknowledges partial support by project PRIMUS/20/SCI/002 from Charles University, by Czech Science Foundation GA\v{C}R, grant 21-00420M, by projects UNCE/SCI/022 and GA UK No.\ 742120 from Charles University, and by SVV-2020-260589.}
	
	%    author two information
	\author{Pavlo Yatsyna}
	\address{
 Aalto University, Department of Mathematics and Systems Analysis, P.O. Box 11100, FI-00076, Finland}
	\curraddr{}
	\email{pavlo.yatsyna@aalto.fi}
	\thanks{P.Y. was supported by the project PRIMUS/20/SCI/002 from Charles University.}
	
	%    \subjclass is required.
	\subjclass[2020]{Primary 11E12, 11D85, 11E25, 11E39}
	
	\date{}
	
	\dedicatory{}
	
	\begin{abstract}
		We improve the bound of the $g$-invariant of the ring of integers of a totally real number field, where the $g$-invariant $g(r)$ is the smallest number of squares of linear forms in $r$ variables that is required to represent all the quadratic forms of rank $r$ that are representable by the sum of squares. Specifically, we prove that the $g_{\O_K}(r)$ of the ring of integers $\O_K$ of a totally real number field $K$ is at most $g_{\Z}([K:\Q]r)$. Moreover, it can also be bounded by $g_{\O_F}([K:F]r+1)$ for any subfield $F$ of $K$. This yields a sub-exponential upper bound for $g(r)$ of each ring of integers (even if the class number is not $1$). Further, we obtain a more general inequality for the lattice version $G(r)$ of the invariant and apply it to determine the value of $G(2)$ for all but one real quadratic field. %Old incorrect wording: "Specifically, we prove that the $g_{\O_K}(r)$ of the ring of integers $\O_K$ of a totally real number field $K$ is at most $g_{\O_F}([K:F]r)$ for any subfield $F$ of $K$." 
	\end{abstract}
	
	\maketitle
	
	\section{Introduction}
	
	The \emph{quadratic Waring's problem} or a \emph{new Waring's problem with squares of linear forms} (eponymous with the title of Mordell's paper \cite{Mo30} that initiated the problem around 1930) asks what is the smallest number of squares needed to represent all admissible quadratic forms. For positive semidefinite quadratic forms in $r$ variables over $\Z$, where $r\le 5$, Mordell and Ko \cite{Ko, Mo30, Mo32} proved that $r+3$ squares of linear forms suffice. Observe that quadratic forms that are sums of squares of linear forms are necessary positive semidefinite. However, there exists a (unique) positive definite quadratic form in $6$ variables which is not a sum of squares \cite{Mo37}.
	
	Let $\Sigma_{\Z}(r)$ be the set of quadratic forms in $r$ variables representable by a sum of squares of linear forms over $\Z$. The $\emph{g-invariant}$ $g_{\Z}(r)$ is the smallest natural number such that every form in $\Sigma_{\Z}(r)$ is, in fact, a sum of $g_{\Z}(r)$ squares. The above results now read as $g_{\Z}(r)=r+3$, for $1\leq r \leq 5$. This is a generalisation of Lagrange's four-square theorem, i.e.\ that $g_{\Z}(1)=4$.
	
	The only remaining known value, $g_{\Z}(6)=10$, was determined just in 1997 by Kim and Oh \cite{KO97}. Much effort went into obtaining upper and lower bounds for $g_{\Z}(r)$. For $7 \leq r \leq 20$, explicit bounds are known \cite{KO02, Sa00}. The upper bounds valid for all $r$ improved gradually: from functions growing faster than exponentially \cite{Ic2} through an exponential \cite{KO05} to the currently best $g_{\Z}(r) = O(\mathrm{e}^{(4+2\sqrt2+\ve)\sqrt{r}})$ due to Beli, Chan, Icaza and Liu \cite{BCIL}.
	
	The $g$-invariant can be generalised to $g_R(\,\cdot\,)$ of an arbitrary ring $R$ by %in an obvious way, simply
	replacing forms over $\Z$ by forms over $R$. In particular, the value $\P(R)=g_R(1)$ is the \emph{Pythagoras number} of $R$, much examined both for arbitrary fields (see, for example, \cite{Le}) and for orders of number fields \cite{HH, Kr, KRS, Pe, Ti, Sch}, and, in an influential paper \cite{CDLR}, for affine and local algebras.
	
	For any number field $K$ and $r\ge 3$, we have $g_K(r)=r+3$ \cite[Prop.\ ~3.2]{BLOP}. The exact values of $g_{\Z}(r)$ for $1\le r\le 5$, given above, derive as a straightforward consequence of the fact that every positive definite form of rank $r$ is represented by a form that is in the genus of $I_{r+3}$ (the sum of $r+3$ squares form): Up to rank eight, this genus contains only one equivalence class \cite{Kn}. %(see \cite[Thm.\ ~1.2]{BI} and \cite[106:3]{OMeara}).
	
	For orders $\O$ in number fields which are not totally real, one has $\P(\O) \leq 5$ \cite{Pf} and for maximal orders even $\P(\O_K) \leq 4$ and more generally $g_{\O_K}(r) \leq r+3$ \cite[Sec.\ ~1]{Ic2} thanks to the theory of spinor genera, but the totally real case exhibits radically different behaviour. For totally real number field $K$, $\P(\O_K)$ can be arbitrarily large \cite{Sch}, but it is bounded by a function depending only on the degree $d=[K:\Q]$ \cite{KY}, namely $\P(\O) \leq g_{\Z}(d)$ for any order $\O$ of degree $d$ (the related bound $\P(K)\le g_F([K:F])$ for fields has already appeared in \cite{CDLR}).
	
	The fact that $g_{\O_K}(r)$ is finite was given together with an upper bound in \cite{Ic1, Ic2}. Chan and Icaza have shown that $g_{\O_K}(r) \leq D \mathrm{e}^{\kappa\sqrt{r}}$ for totally real number fields $K$ of class number $1$, where constants $\kappa$ and $D$ depend only on $K$ \cite[Thm.\ ~1.1]{CI}. We improved this bound as follows:
	\begin{theorem} \label{th:g-overZ}
		Let $\O$ be an order in a number field $K$ of degree $d$. Then
		\[
		g_{\O}(r) \leq g_{\Z}(rd).
		\]
		In particular, for every $\ve>0$, there exists a constant $D$, depending only on $\ve$, such that
		\[
		g_{\O}(r) \leq D \mathrm{e}^{(4+2\sqrt2 + \ve)\sqrt{dr}}.
		\]
		%Let $\O$ be an order in a number field $K$ of degree $d$. Then
		%\[
		%g_{\O}(r) \leq g_{\Z}(rd) \leq D \mathrm{e}^{(4+2\sqrt2 + \ve)\sqrt{dr}},
		%\]
		%where the constant $D$ depends only on $\ve$.
	\end{theorem}
	The first inequality directly follows from Theorem \ref{th:main} (1) below. 
	The latter is a consequence of \cite[Thm.\ ~1.1]{BCIL}, that is, the bound $g_{\Z}(r) = O(\mathrm{e}^{(4+2\sqrt2+\ve)\sqrt{r}})$.
	
	%Let us briefly compare this result with \cite[Thm.\ ~1.1]{CI}. Both results give a bound that is exponential in $\sqrt{r}$. However, beside having a simpler proof, our result includes fields with class number other than one (and even non-maximal orders). Also, our bound is uniform for all number fields of the same degree, while previously the constants depended on the chosen field $K$.
	
	%For number fields with class number larger than $1$, the $g$-invariant is still well-defined, but there is a natural alternative notion: Let $G_{\O_K}(r)$ be obtained by replacing quadratic forms by quadratic lattices and suitably rephrasing the condition about a sum of squares of linear forms. 
	
	\smallskip
	
	For number fields with a class number larger than $1$, the $g$-invariant is still well-defined. However, there is a natural alternative: Let $G_{\O_K}(r)$ be obtained by replacing quadratic forms with quadratic lattices and suitably rephrasing the condition about a sum of squares of linear forms. 
	(The precise formulation is in Definition \ref{de:G}; note that the analogous definition for Hermitian lattices is used in \cite{BCIL, Li1, Li2}.) A more encompassing result of this paper is the following:
	\begin{theorem}\label{th:main}
		Let $K \supset F$ be number fields, $[K:F]=d$, and $\O$ any order in $K$ which contains $\O_F$.
		\begin{enumerate}
			\item If $\O$ is a free $\O_F$-module, then $g_{\O}(r) \leq g_{\O_F}(rd)$. In general, $g_{\O}(r) \leq g_{\O_F}(rd+1)$.
			\item $G_{\O_K}(r)\leq G_{\O_F}(rd)$. In particular, $\P(\O_K) \leq G_{\O_K}(1) \leq G_{\O_F}(d)$.
			\item $g_{\O}(r)\leq G_{\O_F}(rd)$. In particular, $\P(\O) \leq G_{\O_F}(d)$.
		\end{enumerate}
	\end{theorem}
	The proof is a direct consequence of the general Theorem \ref{th:main_general} (the only exception is the second part of (1), which is precisely Corollary \ref{co:rd+1}). In fact, we show that $g_R(r) \leq g_S(dr)$ for a ring extension $R/S$ generated as an $S$-module by $d$ elements. 
	
	Let us compare the statements: If $\O=\O_K$, (3) follows from (2), whereas (1) and (2) are independent. For non-maximal orders, (2) cannot be applied, and (1) and (3) are independent. Finally, the simplest situation when (1) can be applied is if $F$ has the class number $1$; in that case $G_{\O_F}=g_{\O_F}$, so (3) gives exactly the same.
	
	%All the other results in this paper are corollaries of this statement. Since the three variants of the inequality may seem confusing, we now compare them:
	
	%Often, the three parts of the statement give the same result. The simplest situation when (1) can be applied is when the class number of $F$ is one; however, in that case $G_{\O_F}(\,\cdot\,)=g_{\O_F}(\,\cdot\,)$, so (3) gives exactly the same.
	
	%The third statement is arguably the least important one, since it is only interesting for non-maximal orders $\O$. For maximal orders, it follows from (2). In contrast, the statement of (1) is important and new even if $\O=\O_K$.
	
	Let us point out that, aside from providing a first inequality of this type for the $g$-invariants where $r>1$, our Theorem \ref{th:main} also provides a new upper bound for the Pythagoras number in the case when $\O_K$ is not a free $\O_F$-module, thus essentially resolving a question posed in a remark after \cite[Prop.\ ~7.5]{KRS}. This is also evidence that the lattice version of the invariant is important. Further evidence towards its significance is that this invariant is necessary in the proof of $g_{\O}(r) \leq g_{\O_F}(rd+1)$, see Corollary \ref{co:rd+1}. Also, note that in Theorem \ref{th:g-overZ}, if $\O=\O_K$, then $g_{\O_K}(r)$ can be replaced by $G_{\O_K}(r)$. For that, one applies Theorem \ref{th:main} (2) instead of (1). 
	
	Some explicit upper bounds for $g_{\O_F}(r)$ for $r>1$ and a totally real field $F \neq \Q$ were given by Sasaki \cite{SaJapan}. For $F=\Q(\!\sqrt5)$ he proved $g_{\O_F}(2)=5$ (see also \cite[Thm.\ ~7.7]{KRS}), and also showed $g_{\O_{F}}(3) \leq 70$, $g_{\O_{F}}(4) \leq 776$ and $g_{\O_{F}}(5) \leq 3080$. Our results improve these latter bounds to $10$, $37$ and $68$, respectively, since $g_{\Z}(6)=10$, while $g_{\Z}(8)\leq 37$ and $g_{\Z}(10)\leq 68$ \cite{KO02}.
	
	Another immediate consequence of Theorem \ref{th:main} is that $g_{\O}(r)$ is always finite. 
	
	\begin{corollary} \label{co:finite}
		All $g$-invariants of an order in a number field are finite.
	\end{corollary}
	
	Even for rings of integers, this corollary is of interest -- the original proof of finiteness of $g_{\O_K}(r)$ for every $r$ is due to Icaza \cite{Ic1, Ic2} and depends on the results from \cite{HKK} about $\O_K$-lattices, while our approach only relies on \cite{HKK} for the finiteness of $g_{\Z}(\,\cdot\,)$.
	
	It seems that Theorem \ref{th:main} provides optimal upper bounds for most of the cases that we can check. In the case of Pythagoras number ($r=1$), this was illustrated for quadratic fields \cite{Pe}, simplest cubic fields \cite{Ti} and biquadratic fields \cite{KRS}. In this paper, we show that for all but three quadratic fields $K$, the inequality $G_{\O_K}(2) \leq g_{\Z}(4) = 7$ is, in fact, equality.
	
	\begin{theorem}\label{th:quadratic}
		Let $K$ be a real quadratic field other than $\Q(\!\sqrt2)$, $\Q(\!\sqrt3)$, $\Q(\!\sqrt5)$. Then $G_{\O_K}(2) = 7$.
	\end{theorem}
	
	For $\Q(\!\sqrt5)$, the correct value is $5$, as shown by Sasaki \cite{SaJapan}. The same was recently proven for $\Q(\!\sqrt2)$ by He and Hu \cite{HH}; both papers are based on the local-global principle for sums of four integral squares. For the only remaining field $\Q(\!\sqrt3)$ we expect $G(2)=6$, but it can be very difficult to prove, since already for the form $x^2+y^2+z^2$ it is not clear whether the local-global principle holds (for representation of integral binary forms). 
	
	The proof of Theorem \ref{th:quadratic} is contained in Section \ref{se:quadratic}. The upper bound is already clear. For the lower bound in cases $\Q(\!\sqrt{n})$ with $n \equiv 1 \pmod4$, we explicitly produce a quadratic form which is not a sum of $6$ squares of linear forms (Proposition \ref{pr:1mod4}), in fact proving the stronger result $g_{\O_K}(2)=7$. If $n \not\equiv 1 \pmod4$, we use the inequality between $G$-invariants in the other direction, exploiting the results on the Pythagoras numbers of orders in biquadratic number fields from \cite{KRS}.
	
	From the definition, it is clear that $g_{\O_K}(r) \leq G_{\O_K}(r)$, and in Proposition \ref{pr:G<g}, we shall see $G_{\O_K}(r) \leq g_{\O_K}(r+1)$. It is natural to ask: Can $g_{\O_K}(r) < G_{\O_K}(r)$ happen for some number field $K$ and some $r \in \N$, or does equality always hold? Guessing the answer is difficult since the exact values of the invariants are rarely known. However, we fully solve it at least for the ``lattice Pythagoras number'' $G_{\O_F}(1)$ of a quadratic ring of integers:
	
	\begin{theorem}
		Let $F$ be a real quadratic field. Then $\P(\O_F)=G_{\O_F}(1)$. This value is five except for $F = \Q(\!\sqrt2)$, $\Q(\!\sqrt3)$ and $\Q(\!\sqrt5)$, where it is three, and for $F = \Q(\!\sqrt6)$ and $\Q(\!\sqrt7)$, where it is four.
	\end{theorem}
	\begin{proof}
		All the values $\P(\O_F)$ for a real quadratic field $F$ are known thanks to Peters, Maa\ss, Dzewas, Cohn and Pall and an unpublished result by Kneser (contained in Scharlau's dissertation) \cite{Pe, Ma, Dz, CP, Sch2}; an overview of these results can be found in Section 3 of \cite{KRS}. It remains to determine $G_{\O_F}(1)$. All the five exceptional cases have class number $1$, so there is no difference between $\P(\O_F)$ and $G_{\O_F}(1)$. For the other cases, we get $5 = \P(\O_F) \leq G_{\O_F}(1) \leq g_{\Z}(2) =5$.
	\end{proof}
	
	\section{Preliminaries}\label{se:prelim}
	%We use standard terminology about number fields \cite{Na}, and quadratic forms and lattices \cite{OMeara}. 
	We use \cite{Na} as a reference for everything related to number fields and \cite{KS} for quadratic forms and lattices.
	In order to be able to conduct all proofs in the more geometric language of quadratic lattices instead of switching back and forth between lattices and polynomials, we define quadratic lattices over any integral domain (and free lattices even over any commutative ring) instead just over a Dedekind domain.
	%, following %the simple approach of Kneser \cite{KS}.% which generalizes the definitions of O'Meara \cite{OMeara}.
	%For completeness, we include the necessary definitions and facts.
	%In subsection \ref{ss:polynomials}, we briefly review the definition of quadratic forms as polynomials and the resulting notion of the $g$-invariant. 
	
	\subsection{Quadratic forms, \texorpdfstring{$g$}{g}-invariants} \label{ss:polynomials}
	%First, we introduce quadratic forms as polynomials, since this leads to the usual definition of the $g$-invariants and the Pythagoras number. However, this is only done for the sake of completeness, since in this paper, we will instead work in the equivalent setting of free lattices.
	
	Let $R$ be any commutative ring with unity. A \emph{quadratic} (\emph{linear}, resp.) \emph{form} over $R$ in $r$ variables is a homogeneous polynomial in $R[X_1, \ldots, X_r]$ of degree $2$ ($1$, resp.). A form in $r$ variables is also called $r$-ary: unary, binary, ternary, etc. If for quadratic forms $\phi$ and $\psi$ in $r$ and $s$ variables, $r \ge s$, it is possible to find linear forms $L_1, \ldots, L_r$ such that \[\phi\bigl(L_1(X_1, \ldots, X_s), \ldots, L_r(X_1, \ldots, X_s)\bigr) = \psi(X_1, \ldots, X_s),\] we say that $\phi$ \emph{represents} $\psi$. Two forms in the same number of variables which represent each other are called \emph{equivalent}.
	
	Consider the set $\Sigma_R(r)$ of all $r$-ary quadratic forms over $R$ which can be written as a sum of squares of linear forms over $R$, i.e.\ which are represented by the quadratic form $X_1^2 + \cdots + X_N^2$ for some $N\in \N$. Then we put
	\[
	g_R(r) = \min\{n : X_1^2 + \cdots + X_n^2 \text{ represents all forms in } \Sigma_R(r)\}.
	\]
	If no such $n$ exists, we put $g_R(r) = \infty$; however, if $R$ is the ring of integers in a number field, then $g_R(r)$ is finite for every $r$ \cite{Ic2}. (Our Corollary \ref{co:finite} extends this result to non-maximal orders.) Note that $g_R(1) = \P(R)$ is the \emph{Pythagoras number}: The smallest number $P$ such that, if $\alpha \in R$ is a sum of squares, then it can be written as a sum of at most $P$ squares.
	
	%Sometimes we use the word \emph{length} for the minimal number of squares which is necessary to represent a form $\phi$ (or a number $\alpha$).
	
	By \emph{length} we mean the minimal number of squares which is necessary to represent a form $\phi$ (or a number $\alpha$). The length is denoted by $\ell(\,\cdot\,)$ or $\ell_R(\,\cdot\,)$. Hence, $g_R(r)$ is the biggest finite length of an $r$-ary form over $R$.
	
	\subsection{Quadratic spaces and modules}
	Let $R$ be a commutative ring with unity. A \emph{quadratic module} over $R$ is a pair $(M,Q)$, where $M$ is an $R$-module and $Q: M \to R$ is a \emph{quadratic map}, i.e.\ a map such that $Q(a\vect{x})=a^2Q(\vect{x})$ for every $a\in R$ and $\vect{x}\in M$, and that the induced map $B_Q(\vect{x},\vect{y})=Q(\vect{x}+\vect{y})-Q(\vect{x})-Q(\vect{y})$ is bilinear. If $R$ is a field, then $M$ is a vector space and $(M,Q)$ is called a \emph{quadratic space}.
	
	Consider two quadratic modules $(M,Q_M)$ and $(N,Q_N)$ over $R$. An \emph{isometry} is an injective $R$-linear map $\iota: M\to N$ which respects the quadratic maps, i.e.\ $Q_N(\iota(\vect{x}))=Q_M(\vect{x})$ for every $\vect{x}\in M$. If this map is bijective, then the corresponding quadratic modules are \emph{isometric}; this is an equivalence relation denoted by $\simeq$. By omitting the condition of injectivity, we get the notion of \emph{representation}: $M$ is represented by $N$ if there exists any $R$-linear map $\iota: M \to N$ such that $Q_N(\iota(\vect{x}))=Q_M(\vect{x})$ for $\vect{x} \in M$. %This fact is suggestively denoted by $M \to N$.
	This is denoted by $M \to N$ (and by $M \not\rightarrow N$ if such an $R$-linear map for $M$ to $N$ does not exist, i.e. $N$ does not represent $M$).
	
	The \emph{orthogonal sum} $M \perp N$ of the quadratic modules $M$, $N$ is the direct sum of $R$-modules $M \oplus N$ equipped with the quadratic map $Q(\vect{x} + \vect{y}) = Q_M(\vect{x}) + Q_N(\vect{y})$ for $\vect{x}\in M, \vect{y}\in N$.

	\subsection{Free lattices}
	If $L$ is a finitely generated free module of rank $r$ (i.e.\ isomorphic to $R^r$ as an $R$-module) equipped with any quadratic map $Q$, we call $(L,Q)$ a \emph{free quadratic lattice} of \emph{rank} $r$. A lattice of rank $r$ is also called $r$-ary: unary, binary, etc. The free unary lattice $R\vect{e}$ where $Q(\vect{e})=a$ is denoted by $\uqf{a}$; it is unique up to isometry. The basic quadratic module defined over any commutative ring $R$ is $I_n$, which is the free lattice $R^n$ equipped with the ``sum-of-squares'' form $Q(\vect{x}) = \vect{x}^{\mathrm{T}}\vect{x}$. One can also write $I_n \simeq \underbrace{\uqf{1} \perp \cdots \perp \uqf{1}}_{\text{$n$-times}}$.
	
	For a quadratic form $\phi(X_1, \ldots, X_r)$ over $R$, one can construct the corresponding free quadratic lattice $L_{\phi}$ as the $R$-module $R^r$ (with standard basis vectors $\vect{e_1}, \ldots, \vect{e}_r$) equipped with the quadratic map $Q$ defined as $Q(\alpha_1\vect{e}_1 + \cdots + \alpha_r\vect{e}_r) = \phi(\alpha_1, \ldots, \alpha_r)$. On the other hand, if a free quadratic lattice (or in fact any quadratic module) is generated by $\vect{x}_1, \ldots, \vect{x}_k$ as an $R$-module, then they can be used to define a quadratic form $\phi(X_1, \ldots, X_k) = Q(X_1\vect{x}_1 + \cdots + X_k\vect{x}_k)$. In particular, there is a bijection between quadratic forms over $R$ in $r$ variables (up to equivalence) and \emph{free} quadratic lattices of rank $r$ (up to isometry). Also, note that $\phi$ represents $\psi$ if and only if the corresponding lattice $L_\phi$ represents $L_{\psi}$. Especially, $\phi$ can be written as a sum of $n$ squares of linear forms if and only if $L_{\phi}$ is represented by $I_n$.

	\subsection{Quadratic lattices} \label{ss:lattices}
	Assume now that $R$ is an integral domain with the quotient field $F$. We will define quadratic lattices over $R$ as a particularly well-behaved class of quadratic modules. Note that if $R$ is not an integral domain, we only have the notion of a \emph{free} quadratic lattice. Also, for Dedekind domains we will use the much nicer description of lattices given in the next subsection.
	
	Consider a finite-dimensional vector space $V$ over $F$. An \emph{$R$-lattice} in $V$ is any $R$-submodule $L \subset V$ which satisfies $L \subset R\vect{v}_1 + \cdots + R\vect{v}_r$ for some basis $(\vect{v}_1, \ldots, \vect{v}_r)$ of $V$. The \emph{rank} of $L$ is the dimension of the vector space $FL$. If $(V,Q)$ is a finite-dimensional quadratic space over $F$, then a \emph{quadratic lattice} is $(L, Q')$ where $L$ is an $R$-lattice in $V$ and the quadratic map $Q'$ is the restriction of $Q$ to $L$.
	
	Often we do not make a distinction between a quadratic lattice and its underlying $R$-lattice; e.g.\ the \emph{rank} of a quadratic lattice is simply the rank of the corresponding $R$-module. From now on, a lattice usually means a quadratic lattice. Also, we sometimes denote the quadratic maps corresponding to different lattices by $Q$, since it is always clear what the underlying lattice is.
	
	Note that free lattices (over an integral domain) are indeed lattices in this more general sense, since they can be embedded in a quadratic space by tensoring with the quotient field $F$. On the other hand, while lattices (and thus also the lattice version of the $g$-invariants) can be defined over any integral domain, it is much easier to work with them over Dedekind domains, see the next subsection.
	
	Finally, we are ready to define the lattice version of the $g$-invariant:
	
	\begin{definition}\label{de:G}
		For any integral domain $R$, consider the family $\Sigma^{\mathrm{lat}}_R(r)$ of all quadratic lattices of rank $r$ which are represented (over $R$) by $I_N$ for some $N \in  \N$. Then we put
		\[
		G_R(r) = \min\{n : I_n \text{ represents all lattices in } \Sigma^{\mathrm{lat}}_R(r)\}.
		\]
		If no such $n$ exists, we put $G_R(r)=\infty$.
	\end{definition}
	
	To compare, it is clear that
	\[
	g_R(r) = \min\{n : I_n \text{ represents all \emph{free} lattices in } \Sigma^{\mathrm{lat}}_R(r)\};
	\]
	therefore, $g_R(r) \leq G_R(r)$. If $R$ is a PID we have an equality as all lattices are free.%with equality whenever $R$ is a PID.
	
	\smallskip
	
	The following simple lemma is useful since it allows us to work with a smaller family of lattices than the whole $\Sigma^{\mathrm{lat}}_R(r)$.
	
	\begin{lemma}\label{le:altdef}
	For a quadratic lattice $\Lambda$ over an integral domain $R$, the following are equivalent:
	 \begin{enumerate}
	     \item $\Lambda$ represents every lattice $L$ of rank $r$ such that $L \to I_N$ for some $N \in \N$.
	     \item $\Lambda$ represents every sublattice of $I_N$ of rank at most $r$ for every $N \in \N$.
	     \item $\Lambda$ represents every sublattice of $I_N$ of rank $r$ for every $N \in \N$.
	 \end{enumerate}
	Subsequently, Definition \ref{de:G} has the following formulation:
	\[G_R(r)=\min\{n : I_n \text{ represents all sublattices of $I_N$ of rank $r$ for every $N\in \N$}\}.\] 
	%In particular, in the definition of $G_R(r)$ one can replace $\Sigma^{\mathrm{lat}}_R(r)$ by the family of all rank $r$ sublattices of all $I_N$.
	\end{lemma}
	\begin{proof}
	(3) $\implies$ (2): Let $L \subset I_N$ be of rank $s \leq r$. Then $L \perp I_{r-s} \subset I_{N+r-s}$ has rank $r$ and is therefore represented by $\Lambda$. By restricting to $L$, we obtain the required representation.
	
	(2) $\implies$ (1): Let $L$ be a lattice of rank $r$ with a representation $\iota: L \to I_N$. Then $\iota(L)$ is a sublattice of $I_N$ of rank at most $r$, and is therefore represented by $\Lambda$. Thus $L \to \iota(L) \to \Lambda$.
	
	(1) $\implies$ (3): Every sublattice of $I_N$ is represented by $I_N$ and thus by $\Lambda$.
	\end{proof}
	
% 	\begin{observation} \label{ob:altdef}
% 		One could replace $\Sigma^{\mathrm{lat}}_R(r)$ by $\Sigma^{\mathrm{lat}}_R(1) \cup \cdots \cup \Sigma^{\mathrm{lat}}_R(r)$ in the definition of $G_R(r)$. Indeed, it is easily seen that if a lattice represents every lattice in $\Sigma^{\mathrm{lat}}_R(r)$, it also represents all lattices in $\Sigma^{\mathrm{lat}}_R(s)$, $s \leq r$.
		
% 		Going in the other direction, $\Sigma^{\mathrm{lat}}_R(r)$ can also be replaced by the smaller set containing only those lattices of rank $r$ which are sublattices of $I_N$ for some $N$. This is possible since any lattice which represents every sublattice of $I_N$ of rank $r$ also represents all lattices of rank $r$ which are represented by $I_N$. We sometimes implicitly use either of these two reformulations in our proofs.
% 	\end{observation}
	
	Our definitions are valid even for rings of characteristic $2$. However, in that situation the questions considered in this paper are trivial because of the following observation.
	
	\begin{observation} \label{ob:char2}
	Let $2 = 0$ in an integral domain $R$. Then $G_R(r)=g_R(r)=1$ for every $r$. (If $R$ contains zero divisors, we still have $g_R(r)=1$ while $G_R(r)$ is not defined.) This follows from the fact that in characteristic two, $I_n \to I_1$ for every $n$, so every lattice represented by $I_n$ is represented by $I_1$ as well. The representation $\iota: I_n \to I_1$ is $\iota\bigl(\sum\alpha_i\vect{e}_i\bigr) = \bigl(\sum\alpha_i\bigr)\vect{e}$ where $\vect{e_1},\ldots,\vect{e}_n$ and  $\vect{e}$ are the standard bases.
	\end{observation}

	\subsection{Number fields, orders, Dedekind domains}
	Let $K$ be a number field with the ring of integers $\O_K$. Its degree over any subfield $F$ is denoted by $[K:F]$. If $[K:\Q]=d$, then $\O_K$ is a free $\Z$-module of rank $d$, and any basis of this module is called \emph{integral basis} of $\O_K$. An \emph{order} in $K$ is any subring $\O \subset \O_K$ which is also a $\Z$-module of rank $d$. In particular, $\O_K$ is the maximal order with respect to inclusion.
	
	A number field $K$ is called \emph{totally real} if all its embeddings into $\C$ actually map it into $\R$. With one exception, our results and proofs work for all number fields, but they are only interesting for the totally real ones, since $g_{\O_K}(r)\le r+3$ for every $K$ which is not totally real.
	
	Most of the time, the only property of the ring of integers $\O_K$ which we need is the fact that it is a \emph{Dedekind domain}. %Therefore, following the approach of \cite{OMeara}, we usually work over general Dedekind domains of characteristic other than $2$.
	If $R$ is a Dedekind domain, then $R$-lattices are nothing else than finitely generated torsion-free $R$-modules. By the structure theorem \cite[Ch.\ ~1, Thm.\ ~1.32]{Na}, any $R$-lattice $L$ of rank $d$ can be written as a direct sum $\pa_1\vect{x}_1 \oplus \cdots \oplus \pa_d\vect{x}_d$ where $\pa_i$ are fractional ideals and $\vect{x}_i \in L$ (although sometimes it is useful to take them in the vector space $F \cdot L$ where $F$ is the quotient field of $R$).
	
	We shall need the following properties, which follow from the structure theorem:
	\begin{lemma} \label{le:Dedekind}
		Let $R$ be a Dedekind domain.
		\begin{enumerate}
			\item Every lattice $L$ of rank $r$ over $R$ can be written in the form $L = R\vect{x}_1 + \cdots + R\vect{x}_{r-1} + \PA^{-1}\vect{x}_r$, where $\vect{x}_j \in L$ and $\PA$ is an integral ideal in $R$.
			\item Let $S \subset R$ be another Dedekind domain such that $R$ is a torsion-free $S$-module of rank $d$. Then: Any ideal $\PA$ in $R$ can be written as $S\gamma_1 + \cdots + S\gamma_{d-1} + \pb^{-1}\gamma_d$ for $\gamma_i \in \PA$ and $\pb$ an integral ideal in $S$. In particular, $R = S\beta_1 + \cdots + S\beta_{d-1} + \pa^{-1}\beta_d$ for $\beta_i\in R$ and an integral ideal $\pa$ in $S$.
		\end{enumerate}
	\end{lemma}
	
	The most typical situation when (2) applies is if $K \supset F$ are number fields with $d=[K:F]$ and $R=\O_K$, $S=\O_F$. It is important to note that if $\O_F \neq \Z$, there is not necessarily an integral basis of $\O_K$ over $\O_F$ (this happens if and only if $\pa = \O_F$), but there still exists the above \emph{pseudo-basis} $(\beta_1, \ldots, \beta_d)$. 
	
	\smallskip
	
	As mentioned, over a Dedekind domain, $R$-lattices of rank $r$ are exactly the $R$-modules of the form $L=\pa_1 \vect{x}_1 + \cdots + \pa_r \vect{x}_r$ where $\pa_i$ are fractional ideals in $R$ and $\vect{x}_i \in L$ are linearly independent in the vector space $F \otimes_R L$ where $F$ is the quotient field. A quadratic lattice is then any quadratic module on such an $R$-lattice. It is beneficial to consider this the definition of a quadratic lattice since the original definition from Subsection \ref{ss:lattices} is more difficult to work with. This lattice is free if and only if $\pa_1\cdots\pa_r$ is a principal ideal \cite[Ch.\ ~1, Thm.\ ~1.32]{Na}.
	
	If $S \subset R$ are two Dedekind domains and $L = \pa_1x_1 + \cdots + \pa_rx_r$ is a quadratic $S$-lattice (with quadratic map $Q$), we can (and often will) ``extend the scalars'' by taking the tensor product: $R \otimes_S L$ is the quadratic $R$-lattice $R\pa_1x_1 + \cdots + R\pa_rx_r$ (where the quadratic map $Q_R$ extends $Q$).%, see, for example, \cite[\textsection 53]{OMeara}).
	
	Note that a non-maximal order $\O$ is never a Dedekind domain. In particular, although we defined $G_{\O}(r)$, we will prove almost nothing nontrivial about it (only Proposition \ref{pr:grows}) since the theory of lattices over non-maximal orders is much more involved than over Dedekind domains.

	\section{Observations about \texorpdfstring{$G$}{F}}\label{se:observations}
	
	One simple property of the classical $g$-invariant is $g_{\O}(r+1) \ge g_{\O}(r)+1$ over any totally real order $\O$ (see below and compare to \cite[Cor.\ ~2.4]{BLOP}). We show that the same holds for the lattice version $G_{\O}$ as well.
	
	\begin{proposition} \label{pr:grows}
		Let $K$ be a totally real number field and $\O \subset \O_K$ any order. Then:
		\begin{enumerate}
			\item $G_{\O}(r) > G_{\O}(s)$ for $r > s$.
			\item $G_{\O}(r)-r \ge G_{\O}(s)-s$ for $r\ge s$.
			\item $ G_{\O}(r)\ge G_{\O}(1)+r-1$ for all $r\ge 1$.
		\end{enumerate}
		The same is true if we replace $G_{\O}$ by $g_{\O}$.
	\end{proposition}
	\begin{proof}
		We show the claims about $G_{\O}$. The proofs for $g_{\O}$ are verbatim the same, except that every occurrence of ``lattice'' is replaced by ``free lattice''.
		
		Now we explain that it is enough to show $G_{\O}(r+1) \ge G_{\O}(r)+1$ for every $r$. %Indeed: First of all, t
		This inequality immediately gives $(1)$. Further, $G_{\O}(r+1)\ge G_{\O}(r)+1$ can be rearranged as $G_{\O}(r+1)-(r+1)\ge G_{\O}(r)-r$, and recursively, we get $(2)$. Finally, letting $s=1$ in $(2)$ gives us $(3)$.
		
		To show $G_{\O}(r+1) \ge G_{\O}(r)+1$, we shall use the fact that the only way to express $1$ as a sum of squares in $\O$ is $1 = (\pm1)^2$. (Generally, $1$ can never be written as a sum of two totally positive algebraic integers.)
		
		Let $G=G_{\O}(r)$ and take any lattice $L$ (with quadratic map $Q$) of rank $r$ such that $L \to I_N$ for some $N\in \N$, but $L \not\rightarrow I_{G-1}$. Consider now the lattice $L \perp \uqf{1}$ of rank $r+1$. Clearly $L \perp \uqf{1} \to I_{N+1}$. We claim that $L \perp \uqf{1} \not\rightarrow I_G$.
		
		Consider any representation $\iota: L \perp \uqf{1} \to I_G$. Denote by $\vect{f}$ a generating vector of $\uqf{1}$. Then $\iota(\vect{f})$ is a vector in $I_G$ such that $Q(\iota(\vect{f}))=1$. The only such vectors are $\pm$ vectors of the standard basis, as $1 = (\pm1)^2$ is the only way to write $1$ as a sum of squares. Let us say $\iota(\vect{f})=\vect{e}_G$. Then $\iota$ maps $L$ into the orthogonal complement of $\vect{e}_G$. This is a representation $L \to I_{G-1}$ and thus a contradiction.
	\end{proof}

The following proposition may not be true for non-maximal orders, since it significantly uses the properties of lattices over Dedekind domains.

	\begin{proposition}\label{pr:G<g}
		For any Dedekind domain $R$ we have $g_R(r) \leq G_R(r) \leq g_R(r+1)$.
	\end{proposition}
	\begin{proof}
		The first inequality is obvious from the definition. For the second, denote $g=g_R(r+1)$ and consider any $L$ of rank $r$ which is a sublattice of $I_N$ for some $N \in \N$. By Lemma \ref{le:altdef}, it is enough to show that $L \to I_g$.
		
		By assumption, $L$ is a sublattice of $I_N$ of rank $r$:
		\[
		L = R\vect{x}_1 + \cdots + R\vect{x}_{r-1} + \pa^{-1}\vect{x}_r
		\]
		with $\vect{x}_j \in I_N$ and $\pa$ an integral ideal in $R$.
		
		Denote $A = \pa I_1$; clearly it is a unary sublattice of $I_1$. Put $L_1 = L \perp A$; this is a quadratic sublattice of $I_{N+1}$. It is a free lattice of rank $r+1$, so by the definition of $g$ we have $L_1 \to I_g$. Restriction to $L$ yields a representation $L \to I_g$.
	\end{proof}
% 	\begin{proof}
% 		The first inequality is obvious from the definition. For the second, denote $g=g_R(r+1)$ and consider a general quadratic $R$-lattice $L$ of rank $r$ for which there is a representation $\iota: L \to I_N$ for some $N \in \N$. We need to show that $L \to I_g$.
		
% 		The image of $L$ under $\iota$ is a sublattice of $I_N$ of rank $s \leq r$:
% 		\[
% 		\iota(L) = R\vect{x}_1 + \cdots + R\vect{x}_{s-1} + \pa^{-1}\vect{x}_s
% 		\]
% 		with $\vect{x}_j \in I_N$ and $\pa$ an integral ideal in $R$.
		
% 		Denote $A = \pa I_1$; clearly it is a unary sublattice of $I_1$. Now define $M = \iota(L) \perp A$, which is a quadratic sublattice of $I_{N+1}$. It is a free lattice of rank $s+1 \leq r+1$, so by the definition\footnote{Rather then by definition, this is by Observation \ref{ob:altdef}. One has to figure out how exactly to formulate that observation.} of $g$ we have $M \to I_g$. Thus $L \to \iota(L) \to M \to I_g$.
% 	\end{proof}
	
	Assuming Theorem \ref{th:main} (3), i.e.\ the inequality $g_{\O}(r) \leq G_{\O_F}(rd)$, this yields the second part of Theorem \ref{th:main} (1) as a simple corollary:
	
	\begin{corollary}\label{co:rd+1}
		Let $K \supset F$ be number fields, $[K:F]=d$, and $\O$ be any order in $K$ containing $\O_F$. Then $g_{\O}(r) \leq g_{\O_F}(rd+1).$
	\end{corollary}
	\begin{proof}
		Theorem \ref{th:main} (3) together with the previous proposition implies $g_{\O}(r) \leq G_{\O_F}(rd) \leq g_{\O_F}(rd+1)$.
	\end{proof}
	
	Let us stress once again that in most extensions $K/F$, this is the best (to our knowledge) available inequality between $g_{\O_K}(\,\cdot\,)$ and $g_{\O_F}(\,\cdot\,)$ since $\O_K$ usually does not have an integral basis over $\O_F$, which makes it impossible to use the first part of Theorem \ref{th:main} (1).
	
	\section{The main proof}\label{se:proof}
	
	In this section we prove the key result of this paper -- the inequalities between the $g$-, resp.\ $G$-invariants of a ring and its subring. This theorem implies most of the other contents of this article. Since the proof does not use any number theoretic properties, we formulate it more generally for commutative rings with unity and for Dedekind domains. Namely, we prove the following:
	
	\begin{theorem}\label{th:main_general}
		Let $R \supset S$ be commutative rings with unity.
		\begin{enumerate}
			\item If $R$ is generated by $d$ elements as an $S$-module, then $g_R(r) \leq g_S(rd)$.
			\item If $R,S$ are Dedekind domains and $R$ is a torsion-free $S$-module of rank $d$, then $G_R(r)\leq G_S(rd)$.
			\item If $S$ is a Dedekind domain and $R$ is a torsion-free $S$-module of rank $d$, then $g_R(r)\leq G_S(rd)$.
		\end{enumerate}
	\end{theorem}
	
	Although we formulate it as one theorem and %although 
	the main idea behind the proofs is the same, the three statements are independent and neither of them implies any of the others. If we replace Dedekind domains by $\O_K$ for a number field $K$ and rings by (not necessarily maximal) orders $\O$, we obtain precisely the main Theorem \ref{th:main}, almost word for word.
	
	We start with part (1); its proof could be written purely in terms of polynomials, but we instead choose the equivalent language of free quadratic lattices. 
	
	\begin{proof}[Proof of Theorem \ref{th:main_general} (1)]
	By assumption, we have $R = S\beta_1 + \cdots + S\beta_d$ for some $\beta_i \in R$. We %will denote 
	let $g = g_S(rd)$. Consider a free lattice $L$ over $R$ of rank $r$ for which there is a representation $\iota: L \to I_N$; our aim is to prove that $L \to I_g$.
	
	Denote the basis vectors of $L$ by $\vect{x}_1, \ldots, \vect{x}_r$. Now, $\iota(L) = \sum_{j \leq r} R \iota(\vect{x}_j)$ is a submodule of $I_N$ over $R$. Every $\iota(\vect{x}_j)$ is an element of $R^N$, so by assumption we can write $\iota(\vect{x}_j) = \vect{f}_1^{(j)}\beta_1 + \cdots + \vect{f}_d^{(j)}\beta_d$ with $\vect{f}_i^{(j)} \in S^N$. (Since $R$ is not necessarily a free $S$-module, these ``vectors of coefficients'' $\vect{f}_i^{(j)}$ are not unique, but they do exist.)
	
	Consider now the quadratic $S$-module
	\[
	M = \sum_{i \leq d}\sum_{j \leq r} S \vect{f}_i^{(j)},
	\]
	which is a quadratic submodule of $I_N$ over $S$ generated by $rd$ elements. If it were free, then by definition of $g$ it is represented by $I_g$ over $S$; however, this is in general not the case. Thus, let us consider the following auxiliary free quadratic $S$-lattice $\widetilde{M}$ of rank $rd$: The underlying $S$-module is $S^{rd}$ and we denote its basis by $\vect{e}_{ij}$. The quadratic map $\widetilde{Q}$ is defined as follows:
	\[
	\widetilde{Q}\Bigl(\sum_{i,j} \alpha_{ij}\vect{e}_{ij}\Bigr) = Q_M\Bigl(\sum_{i,j} \alpha_{ij}\vect{f}_{i}^{(j)}\Bigr),
	\]
	where $Q_M$ is the quadratic map on $M$. Not only is $\widetilde{Q}$ a well-defined quadratic map; it was defined in such a way that the linear map $\widetilde{M} \to M$ given by $\vect{e}_{ij} \mapsto \vect{f}_i^{(j)}$ is a representation. Since $M$ is a subset of $I_N$, this yields a representation $\widetilde{M} \to I_N$ over $S$. Thus, by the definition of $g$, $\widetilde{M} \to I_g$ over $S$. If we extend scalars by taking the tensor product, we get $R \otimes_S \widetilde{M} \to R \otimes_S I_g$, which is usually written as ``$R \otimes_S \widetilde{M} \to I_g$ over $R$''.
	
	In the lattice $R \otimes_S \widetilde{M}$ one can find the vectors $\vect{y}_j = \vect{e}_{1j}\beta_1 + \cdots + \vect{e}_{dj}\beta_d$. Consider now the $R$-linear map $L \to R \otimes_S \widetilde{M}$ given by $\vect{x}_j \mapsto \vect{y}_j$. One easily sees that it respects the quadratic map and thus it is a representation. Hence we have $L \to R \otimes_S \widetilde{M} \to I_g$, which concludes the proof.
	\end{proof}
	We note that the proof does not require the quadratic structure on the modules; it would work just as well for modules equipped with a different map type. In particular, a statement analogous to Theorem \ref{th:main_general} holds for representing cubic forms by sums of cubes, etc. 
	
	\smallskip
	
	%Concerning the second and third part of Theorem \ref{th:main_general}, again, the quadratic structure does not play the main role in the proof. 
	
	The quadratic structure does not play the main role in the proofs of the second and third parts of Theorem \ref{th:main_general}, either. Instead, most of the work consists of handling the underlying $S$-modules and $R$-modules seen as subsets of $\bigl(\FoF{S}\bigr)^n$ or $\bigl(\FoF{S}\cdot R\bigr)^n$ (where $\FoF{\,\cdot\,}$ means the quotient field) and showing that they are in fact subsets of $S^n$ or $R^n$. Note that for an ideal $\pa \subset S$ we write $(\pa S)^n$ to mean $n$-dimensional vectors with coordinates in $\pa$; this is to avoid confusion with powers of ideals.
	
	%While the key idea behind the proofs of part (1) and (2) is the same, the details are different. On one hand, the definitions of lattices in part (2) are more complicated since one has to work with non-principal ideals. 
	The definitions of lattices in part (2) are more complicated than in (1) since one has to work with non-principal ideals. On the other hand, we can avoid using the auxiliary lattice $\widetilde{M}$ thanks to %the simple 
	Lemma \ref{le:altdef}.

	\begin{proof}[Proof of Theorem \ref{th:main_general} (2)]
		Consider a quadratic $R$-lattice $L$ of rank $r$ which is a sublattice of $I_N$ for some $N \in \N$. By Lemma \ref{le:altdef}, it is enough to show $L\to I_G$ over $R$, where $G=G_S(rd)$. By Lemma \ref{le:Dedekind} (1), $L = R \vect{x}_1 + \cdots + R\vect{x}_{r-1} + \PA^{-1}\vect{x}_r,$
		where each $\vect{x}_i$ is an element of $R^N$ and $\PA$ is an integral ideal. Observe that $\PA^{-1}\vect{x}_r \subset L$ implies $\vect{x}_r \in (\PA R)^N$.
		
		Lemma \ref{le:Dedekind} (2) provides us with the pseudo-bases $(\beta_i)_{i\leq d}$ of $R$ and $(\gamma_i)_{i\leq d}$ of $\PA$, together with integral ideals $\pa$ and $\pb$ in $S$, such that
		\begin{align*}
			R &= S\beta_1 + \cdots + S\beta_{d-1} + \pa^{-1}\beta_d \qquad\text{ with $\beta_i\in R$,}\\
			\PA &= S\gamma_1 + \cdots + S\gamma_{d-1} + \pb^{-1}\gamma_d \qquad\text{ with $\gamma_i \in \PA$.}
		\end{align*}
		Now define the vectors $\vect{f}_i^{(j)} \in \bigl(\FoF{S}\bigr)^N$ as the coordinates of $\vect{x}_j$ with respect to these pseudo-bases:
		\begin{align*}
			\vect{x}_j &= \vect{f}_1^{(j)}\beta_1 + \cdots + \vect{f}_d^{(j)}\beta_d \qquad \text{ for $1 \leq j \leq r-1$};\\
			\vect{x}_r &= \vect{f}_1^{(r)}\gamma_1 + \cdots + \vect{f}_d^{(r)}\gamma_d.
		\end{align*}
		
		If $i\neq d$, then $\vect{f}_i^{(j)} \in S^N$. For $j\neq r$ one has $\vect{f}_d^{(j)} \in (\pa^{-1}S)^N$ and in the last case $\vect{f}_d^{(r)} \in (\pb^{-1}S)^N$. Therefore
		$M = \sum_{i<d, j \leq r} S\vect{f}_i^{(j)} \,+\,  \sum_{j<r} \pa\vect{f}_d^{(j)} \,+\, \pb\vect{f}_d^{(r)}$
		is a well-defined sublattice of $S^N$ of rank at most $rd$. We can also understand it as a quadratic sublattice of $I_N$ over $S$. By the definition of $G=G_S(rd)$ (or, strictly speaking, by Lemma \ref{le:altdef}), it is represented by $I_G$ over $S$.
		
		Since $M \to I_G$ over $S$, we get $R \otimes_S M \to I_G$ over $R$. On the other hand, one can explicitly write
		\[
		R \otimes_S M = \sum_{i<d,\, j \leq r} R\vect{f}_i^{(j)} \,+\,  \sum_{j<r} (R\pa)\vect{f}_d^{(j)} \,+\, (R\pb)\vect{f}_d^{(r)}.
		\]
		That is, $M$ was defined in such a way that $\vect{x}_j \in R \otimes_S M$ for $j \leq r-1$ and $\PA^{-1}\vect{x}_r \subset R \otimes_S M$ -- to see this, observe $\beta_d \in R\pa$, $\PA^{-1}\gamma_i \subset R$ and $\PA^{-1}\gamma_d \subset  R \pb$. In particular, $L$ is a sublattice of (and thus represented by) $R \otimes_S M$. Composition of these two representations yields $L \to R \otimes_S M \to I_G$ over $R$. This is the desired representation $L \to I_G$.
	\end{proof}
	
	The proof of part (3) is almost the same as for (2) and the technical details are in fact slightly simpler. Therefore our explanations will be less detailed.
	
	\begin{proof}[Proof of Theorem \ref{th:main_general} (3)]
		By assumption, we have $R = S \beta_1 + \cdots + S\beta_{d-1} + \pa^{-1}\beta_d$, where each $\beta_i \in R$ and $\pa$ is an integral ideal in $S$. Clearly, $\pa^{-1} \beta_d \subset R$.

		Denote $G=G_S(rd)$. Consider any free $R$-lattice $L$ in $r$ variables which is represented by $I_N$. We need to show that $L$ is represented by $I_G$ over $R$. Denote by $L'$ the quadratic $R$-submodule of $I_N$ which is the image of $L$ under the representation $L \to I_N$. It is generated by $r$ vectors, say $L' = R\vect{x}_1 + \cdots + R\vect{x}_r$, where all $\vect{x}_j \in R^N$. Since $L \to L'$, it suffices to show that $L' \to I_G$.
		
		Denote $\vect{x}_j = \vect{f}_1^{(j)}\beta_1 + \cdots + \vect{f}_{d-1}^{(j)}\beta_{d-1} + \vect{f}_d^{(j)}\beta_d.$
		It is important to note that while $\vect{f}_i^{(j)} \in S^N$ for $1\leq i \leq d-1$ (and $1\leq j \leq r$), one only gets $\vect{f}_d^{(j)} \in (\pa^{-1} S)^N$.
		
		Define the $S$-lattice $
		M = \sum_{j \leq r} \Bigl( S\vect{f}_1^{(j)}  + \cdots + S\vect{f}_{d-1}^{(j)} + \pa\vect{f}_d^{(j)} \Bigr)$
		of rank at most $rd$. We claim that it is a sublattice of $I_N$ over $S$; that is, it contains only vectors from $S^N$. This only has to be checked for the last summand from each bracket; and there, indeed, $\pa \cdot \vect{f}_d^{(j)} \subset \pa \cdot (\pa^{-1}S)^N \subset S^N$. Therefore, $M$ is a well-defined quadratic sublattice of $I_N$. By the definition of $G = G_S(rd)$ (or more precisely by Lemma \ref{le:altdef}), it is represented by $I_G$ over $S$.
		
		Since $M \to I_G$ over $S$, we get $R \otimes_S M \to I_G$ over $R$. We also claim $\vect{x}_j \in R \otimes_S M$ for every $j$; to see this, remember $\beta_d \in R\pa$. So, $L'$ is a sublattice of (and thus represented by) $R \otimes_S M$. Composition of these two representations yields $L' \to R \otimes_S M \to I_G$ over $R$. This is the desired representation $L' \to I_G$.
	\end{proof}

	\section{\texorpdfstring{$G(2)$}{G(2)} for quadratic rings of integers} \label{se:quadratic}
	While the definition of $g(r)$ may seem more straightforward than the definition of $G(r)$, we consider the latter to be the more useful and natural invariant. Perhaps it is not found in the literature simply because almost nothing nontrivial was known for any field of a class number other than $1$ (see \cite{Li1} for an example for Hermitian lattices). Now we can partly remedy this by deciding the $G(2)$-invariant of most real quadratic fields -- that is, we prove Theorem \ref{th:quadratic}.
	
	Combining our inequalities, we are able to get not only an upper bound but in more than half cases also the lower bound:
	
	\begin{lemma} \label{le:23mod4}
		Let $F = \Q(\!\sqrt{n})$ be a quadratic field, $n>1$ square-free. Then:
		\begin{enumerate}
			\item $G_{\O_F}(2) \leq 7$.
			\item If $n \not\equiv 1 \pmod4$, $n \ge 10$, then $G_{\O_F}(2) =7$.
		\end{enumerate} 
	\end{lemma}
	\begin{proof}
		The first statement is just the inequality $G_{\O_F}(2) \leq g_{\Z}(4) = 7$.
		
		For the second, we need $7 \leq G_{\O_F}(2)$. Pick any square-free $n_2 \ge 10$ coprime to $n$ such that one of $n,n_2$ is $2$ and the other $3$ modulo $4$. By \cite[Thm.\ ~5.3]{KRS}, the biquadratic field $K = \Q(\!\sqrt{n},\sqrt{n_2})$ has $\P(\O_K) = 7$. The extension $K/F$ is of degree $2$, so Theorem \ref{th:main} applies: $7 = \P(\O_K) \leq G_{\O_F}(2)$.
	\end{proof}
	
	To prove Theorem \ref{th:quadratic}, it remains to treat the fields $\Q(\!\sqrt{n})$ for $n=6$, $n=7$ and most importantly for $n \equiv 1 \pmod4$. By a direct computation we get the following:
	\begin{lemma}\label{le:biquadratic}
		Biquadratic fields $K_1= \Q(\!\sqrt6,\sqrt7)$ and $K_2=\Q(\!\sqrt{13},\sqrt{15})$ have $\P(\O_{K_i})=7$. More specifically, the following elements are sums of seven but not of six squares in $\O_{K_i}$:
		\begin{itemize}
			\item $\alpha_1 = 43 + \sqrt{6} - 8\sqrt{7} + \sqrt{7\cdot 6}$;
			\item $\alpha_2 = 114+ 15\sqrt{13} + 20\sqrt{15} + 6\sqrt{13\cdot 15}$. 
		\end{itemize}
		Thus $G_{\O_F}(2)=g_{\O_F}(2)=7$ for $F=\Q(\!\sqrt{n})$, $n=6,7,13$.
	\end{lemma}
	\begin{proof}
		As soon as we check that $\alpha_i$ is a sum of six but not seven squares (i.e.\ its length is $7$), the rest is easy: $\P(\O_{K_i})=7$ will follow from the just obtained $\P(\O_{K_i}) \ge \ell(\alpha_i) = 7$ combined with the inequality $\P(\O_{K_i}) \leq g_{\Z}(4)=7$. After this, we also get $\P(\O_{K_i})\leq G_{\O_F}(2) \leq g_{\Z}(4)$ for $[K_i:F]=2$, so $G_{\O_F}(2)=7$; and since all three quadratic fields have class number $1$, we also have $g_{\O_F} = G_{\O_F}$.
		
		So it remains to explain that $\alpha_i$ is a sum of seven but not of six squares in $\O_{K_i}$, which we computed using Magma \cite{BCP}.
	\end{proof}
	
	On a side note, the lemma also yields $g_{\O_F}(2)=7$ for $F=\Q(\!\sqrt{15})$: Although the class number is $2$ and therefore the equality $g_{\O_F}(2)=G_{\O_F}(2)$ is not immediate, we have $\O_K = \Z[\sqrt{15}] \cdot 1 + \Z[\sqrt{15}] \cdot \frac{1+\sqrt{13}}{2}$, so $\O_K$ is a free $\O_F$-module, implying $7=\P(\O_K) \leq g_{\O_F}(2)$.
	
	\subsection{The case 1 mod 4} \label{ss:1mod4}
	
	For $F=\Q(\!\sqrt{n})$, $n=5$ we have Sasaki's result $G_{\O_F}=g_{\O_F}(2)=5$. The case $n=13$ was handled separately in Lemma \ref{le:biquadratic}. In this subsection we treat all the remaining real quadratic fields $\Q(\!\sqrt{n})$ with $n \equiv 1 \pmod4$.
	
	It is worth mentioning that the form $\phi$ in the following theorem was discovered by examining elements of length seven in real biquadratic fields $\Q(\!\sqrt{p},\sqrt{n})$ for the first few values of $n$ and for $p \not\equiv 1 \pmod4$ coprime with $n$. Quite probably there are other forms with the same property.
	
	\begin{proposition} \label{pr:1mod4}
		If $F=\Q(\!\sqrt{n})$ for $n\ge 17$ square-free, $n \equiv 1 \pmod4$, then $g_{\O_F}(2) \ge 7$. In particular, denote
		\begin{align*}
			\phi(X,Y) &= \Bigl(7 + \bigl(\tfrac{1+\sqrt{n}}{2}\bigr)^2\Bigr)X^2 + \Bigl(\tfrac{5+\sqrt{n}}{2}X + Y\Bigr)^2 + \Bigl((5+\sqrt{n})X + \tfrac{1+\sqrt{n}}{2}Y\Bigr)^2\\
			&= \frac{3n+77+26\sqrt{n}}{2}X^2 + (10+n+7\sqrt{n})XY + \frac{n+5+2\sqrt{n}}{4}Y^2.
		\end{align*}
		Then $\ell(\phi)=7$, i.e.\ $\phi$ is a sum of $7$, but not of $6$ squares of binary linear forms.
	\end{proposition}
	\begin{proof}
		The fact that $\phi$ is a sum of seven squares is clear from the first representation, since $7 + \bigl(\frac{1+\sqrt{n}}{2}\bigr)^2$ is a sum of five squares. So it remains to show that $\phi$ is never a sum of six or less squares.
		
		First we solve the cases $n \leq 53$. We define the biquadratic field $K=\Q(\!\sqrt{10},\sqrt{n})$. One integral basis of $K$ is $\bigl(1,\sqrt{10},\frac{1+\sqrt{n}}{2},\sqrt{10}\frac{1+\sqrt{n}}{2}\bigr)$, so $\O_K = \O_F \cdot 1 + \O_F \cdot\sqrt{10}$. Instead of directly applying the inequality $\P(\O_K)\leq g_{\O_F}(2)$, we use the simple idea from its proof: Put $\alpha_{10} = \phi(1,\sqrt{10})$. By definition, this is a sum of seven squares in $\O_K$; in fact, it is clear that $\ell_{\O_K}(\alpha_{10}) \leq \ell_{\O_F}(\phi)$. By a direct computation (e.g.\ in Magma) as in Lemma \ref{le:biquadratic}, we easily check that $\ell_{\O_K}(\alpha_{10})=7$ for $n=17$, $21$, $29$, $33$, $37$, $41$ and $53$. This means that $\ell_{\O_F}(\phi)\ge 7$ in these cases.
		
		Now we solve the cases $53 < n \leq 65$. The previous approach fails, since $\ell(\alpha_{10})$ turns out to be $5$ for all three $n$ in question. However, we can use the same trick with $10$ replaced by $11$: By the same direct computation we check that $\alpha_{11}=\phi(1,\sqrt{11})$ has length $7$ for $n=57$, $61$ and $65$. (While for $n=73$, this trick again fails, and, surprisingly, even replacing $11$ by $14$, $15$, $19$, $22$, $23$, $26$, $30$ or $31$ is useless.)
		
		Now comes the main part of the proof, for $n \ge 73$. Consider a representation of $\phi$ as a sum of squares of linear forms, i.e.\ $\phi(X,Y) = \sum_i (x_iX + y_iY)^2$ with $x_i,y_i \in \O_F$.
		By comparing the coefficients of $Y^2$ one gets $\sum y_i^2 = \frac{5+n}{4} + \frac{\sqrt{n}}{2}$. The number on the right decomposes uniquely as a sum of squares, namely as $(\pm1)^2 + \bigl(\pm \frac{1+\sqrt{n}}{2}\bigr)^2$; therefore, after possibly transferring the signs to the $x_i$ and reordering the terms, the decomposition must be of the form
		\[
		\phi(X,Y) = (x_1X + Y)^2 + \bigl(x_2X + \tfrac{1+\sqrt{n}}{2}Y\bigr)^2 + \sum_{i>2} (x_iX)^2.
		\]
		Before we proceed, observe that it suffices to show that $x_1=\frac{5+\sqrt{n}}{2}$ and $x_2=5+\sqrt{n}$: Once this is proven, one has $\sum_{i>2} x_i^2 = 7 + \bigl(\frac{1+\sqrt{n}}{2}\bigr)^2$, which requires at least five squares in $\O_F$, see \cite[p.\ ~161]{Pe}.
		
		Comparing the coefficients of $XY$ yields $x_1+x_2\frac{1+\sqrt{n}}{2} = \frac{10+n+7\sqrt{n}}{2}$, which we rearrange as
		\begin{equation} \label{eq:x1}
			x_1 = \frac{10+n+7\sqrt{n}}{2} - \frac{1+\sqrt{n}}{2}x_2.
		\end{equation}
		This means that it only remains to show $x_2 = 5 + \sqrt{n}$.
		
		Let us now focus on the equality $\frac{3n+77+26\sqrt{n}}{2} = \sum x_i^2 = \sum \bigl(\tfrac{a_i+b_i\sqrt{n}}{2}\bigr)^2$ where $a_i\equiv b_i \pmod2$.
		It can equivalently be rewritten as a system of two Diophantine equations (with the above parity condition):
		\begin{align}
			\sum a_i^2 + n\sum b_i^2 &= 154 + 6n,\label{eq:1} \\ 
			\sum a_ib_i &= 26. \label{eq:2}
		\end{align}
		Looking at \eqref{eq:2} modulo $4$, we see that there must be a nonzero even number of indices $i$ such that $a_i \equiv b_i \equiv 1 \pmod2$. In particular, $\sum b_i^2$ is an even number. Also, $\sum a_i^2 \ge 2$.
		
		We shall now prove that $\sum b_i^2 \leq 6$. If not, then $\sum b_i^2 \ge 8$, so \eqref{eq:1} gives $2 + 8n \leq 154 + 6n$, a contradiction for $n \ge 77$. It remains to deal with $n=73$. In this case, $\sum b_i^2 \ge 10$ is impossible, so we have $\sum b_i^2 = 8$; then $\sum a_i^2 = 8$ by \eqref{eq:1}. By Cauchy--Schwarz inequality, $\left|\sum a_i b_i\right| \leq \sqrt{8 \cdot 8} = 8$, which contradicts \eqref{eq:2}.
		
		So we indeed have $\sum b_i^2 \leq 6$ for $n \ge 73$. In particular, $|b_i| \leq 2$ for every $i$.
		
		Rewrite now \eqref{eq:x1} explicitly as $x_1 = \bigl(5 - \frac{a_2}{4} +n \frac{2-b_2}{4}\bigr) + \sqrt{n}\bigl(\frac72 - \frac{a_2+b_2}{4}\bigr)$, i.e.
		$a_1 = 10 - \tfrac{a_2}{2} + n\tfrac{2-b_2}{2}$ and $b_1 = 7 - \tfrac{a_2+b_2}{2}.$
		Since $|b_1|,|b_2| \leq 2$, triangle inequality applied to the second equation gives $|a_2| \leq 14 + |b_2| + |2b_1| \leq 20$. Therefore, the first equation yields $a_1 \ge n\frac{2-b_2}{2}$.
		
		We now show that $b_2=2$: If not, then $b_2<2$, so $a_1 \ge \frac{n}{2}$. However, plugging this in \eqref{eq:1} yields $\frac14 n^2 \leq 154 + 6n$, which is a contradiction for $n \ge 40$.
		
		Having proven $b_2=2$, we are almost done. It only remains to determine the even number $a_2$. Combining $\sum b_i^2 \leq 6$ with the fact that there are at least two indices $i$ such that $b_i$ is odd, one sees that necessarily $\sum b_i^2 = 6$. We have $b_1 = 6 - \frac{a_2}{2}$. Since $|b_1| \leq 1$ and $a_2$ is even, there are only three possibilities for $a_2$.
		
		First assume $a_2=14$. Plugging $\sum b_i^2 = 6$ into \eqref{eq:1}, we get $\sum a_i^2 = 154$, so our value of $a_2$ is impossible.
		
		The second option is $a_2=12$. In this case, $a_1=10 - \frac{12}{2} + n\frac{2-2}{2} = 4,$ so $a_1^2 + a_2^2 = 16 + 144 = 160,$ which again contradicts the equality $\sum a_i^2 = 154$.
		
		The only remaining case is $a_2=10$. This means $x_2 = \frac{a_2+b_2\sqrt{n}}{2} = 5 + \sqrt{n}$, and we get $x_1 = \frac{5+\sqrt{n}}{2}$. Thus $(x_1X+y_1Y)^2$ and $(x_2X+y_2Y)^2$ are exactly the second and third term in the original definition of $\phi$, and we already explained that the first term requires at least five more squares.
	\end{proof}
	%We summarise:
	\begin{proof}[Proof of Theorem \ref{th:quadratic}]
		The upper bound $G_{\O_F}(2) \leq 7$ is in Lemma \ref{le:23mod4} (1). Part (2) of the same Lemma gives the lower bound $G_{\O_F}(2) \ge 7$ for $n \not\equiv 1 \pmod4$, $n \ge 10$. For $n=6, 7, 13$, the lower bound is contained in Lemma \ref{le:biquadratic}, and for the remaining $n \equiv 1 \pmod4$ in Proposition \ref{pr:1mod4}.
	\end{proof}
	
	\section*{Acknowledgements}
	
	We thank the anonymous reviewer for his or her valuable suggestions.
	
	\bibliographystyle{amsplain}

\begin{thebibliography}{plain}
		
		%\bibitem[BI]{BI} R. Baeza and M. I. Icaza, \textit{Decomposition of positive definite integral quadratic forms as sums of positive definite quadratic forms}, In {$K$}-theory and algebraic geometry: connections with quadratic forms and division algebras ({S}anta {B}arbara,{CA}, 1992), \textbf{58} of Proc. Sympos. Pure Math., pages 63--72, Amer. Math. Soc., Providence, RI, (1995).
		
		\bibitem[BLOP]{BLOP} R. Baeza, D. Leep, M. O'Ryan and M. J. P. Prieto, \textit{Sums of squares of linear forms}, Math. Z. 193, no. 2, 297--306, (1986).
		
		\bibitem[BCIL]{BCIL} C. N. Beli, W. K. Chan, M. I. Icaza and J. Liu, \textit{On a Waring’s problem for integral quadratic and Hermitian forms}, Trans. Amer. Math. Soc. 371, 5505--5527 (2019).
		
		\bibitem[BCP]{BCP} W. Bosma, J. Cannon and C. Playoust, \textit{The Magma algebra system. I. The user language}, J. Symbolic Comput., 24,  235–-265 (1997).
		
		\bibitem[CI]{CI} W. K. Chan and M. I. Icaza, \textit{Hermite reduction and a Waring’s problem for integral quadratic forms over number fields}, Trans. Amer. Math. Soc. 374, 2967--2985 (2021).
		
		\bibitem[CDLR]{CDLR} M.D. Choi, Z.D. Dai, T.Y. Lam and B. Reznick, \textit{The Pythagoras number of some affine algebras and local algebras}, J. Reine Angew. Math. 336, 45--82 (1982).
		
		\bibitem[CP]{CP} H. Cohn and G. Pall, \textit{Sums of four squares in a quadratic ring}, Trans. Amer. Math. Soc. 105, 536--556 (1962).
		
		\bibitem[Dz]{Dz} J. Dzewas, \textit{Quadratsummen in reell-quadratischen Zahlkörpern}, Math. Nachr. 21, 233--284 (1960).
		
		\bibitem[HH]{HH} Z. He and Y. Hu, \textit{Pythagoras number of quartic orders containing $\sqrt{2}$}, \href{https://arxiv.org/abs/2204.10468}{arXiv:2204.10468}.
		
		\bibitem[HKK]{HKK} J. S. Hsia, Y. Kitaoka and M. Kneser, \textit{Representations of positive definite quadratic forms} J. Reine Angew. Math. 301, 132--141 (1978).
		
		\bibitem[Ic1]{Ic1} M.I. Icaza, \textit{Effectiveness in representations of positive definite quadratic forms},
		Thesis (Ph.D.)–The Ohio State University. 1992. 76 pp.
		
		\bibitem[Ic2]{Ic2} M. I. Icaza, \textit{Sums of squares of integral linear forms}, Acta Arith. 124, 231--241 (1996).
		
		\bibitem[Kn]{Kn} M. Kneser, \textit{Klassenzahlen definiter quadratischer Formen}, Arch. Math. (Basel) 8, 241--250 (1957).
		
		\bibitem[Ko]{Ko} C. Ko, \textit{On the representation of a quadratic form as a sum of squares of linear forms}, Q. J. Math. 1, 81--98 (1937).
		
		\bibitem[KO1]{KO97} M.-H. Kim and B.-K. Oh, \textit{Representations of positive definite senary integral quadratic forms by a sum of squares}, J. Number Theory 63, 89--100 (1997).
		
		\bibitem[KO2]{KO02} M.-H. Kim and B.-K. Oh, \textit{Bounds for quadratic Waring’s problem}, Acta Arith. 104, 155--164 (2002).
		
		\bibitem[KO3]{KO05} M.-H. Kim and B.-K. Oh, \textit{Representations of integral quadratic forms by sums of squares}, Math. Z. 250, 427--442 (2005).
		
		\bibitem[Kr]{Kr} J. Krásenský, \textit{A cubic ring of integers with the smallest Pythagoras number}, Arch. Math. (Basel) 118, 39--48 (2022).
		
		\bibitem[KRS]{KRS} J. Kr\'asensk\'y, M. Ra\v{s}ka and E. Sgallov\'a, \textit{Pythagoras numbers of orders in biquadratic fields}, Expo. Math. (2022). \href{https://doi.org/10.1016/j.exmath.2022.06.002}{doi:10.1016/j.exmath.2022.06.002} %Also at: \href{https://arxiv.org/abs/2105.08860}{arXiv:2105.08860}.
		
		\bibitem[KS]{KS} M. Kneser and R. Scharlau, \textit{Quadratische Formen}, Springer (2002).
		
		\bibitem[KY]{KY} V. Kala and P. Yatsyna, \textit{Lifting problem for universal quadratic forms}, Adv. Math. 377, 24 pp. (2021).
		
		\bibitem[Le]{Le} D. Leep, \textit{A historical view of the {P}ythagoras numbers of fields}, In Quadratic forms---algebra, arithmetic, and geometry, \textbf{493} of Contemp. Math., pages 271--288. Amer. Math. Soc., Providence, RI (2009).
		
		\bibitem[Li1]{Li1} J. Liu, \textit{g-invariant on unary Hermitian lattices over imaginary quadratic fields with class number 2 or 3}, \href{https://arxiv.org/abs/2111.10825}{arXiv:2111.10825}.
		
		\bibitem[Li2]{Li2} J. Liu, \textit{On a Waring's problem for Hermitian lattices}, Bull. Sci. math. (2021).
		
		\bibitem[Ma]{Ma} H. Maa{\ss}, \textit{{\"U}ber die Darstellung total positiver Zahlen des K{\"o}rpers  R($\sqrt5$) als Summe von drei Quadraten}, Abh. Math. Sem. Univ. Hamburg 14, 185--191 (1941).
		
		\bibitem[Mo1]{Mo30} L. J. Mordell, \textit{A new Waring's problem with squares of linear forms}, Q. J. Math. 1, 276--288 (1930).
		
		\bibitem[Mo2]{Mo32} L. J. Mordell, \textit{On the representation of a binary quadratic form as a sum of squares of linear forms}, Math. Z. 35, 1--15 (1932).
		
		\bibitem[Mo3]{Mo37} L. J. Mordell, \textit{The Representation of a Definite Quadratic Form as a Sum of Two Others}, Annals of Mathematics  38(4), 751--757 (1937). %https://doi.org/10.2307/1968831
		
		\bibitem[Na]{Na} W. Narkiewicz, \textit{Elementary and Analytic Theory of Algebraic Numbers}, Springer (1990).
		
%		\bibitem[O'M]{OMeara} O. T. O’Meara, \textit{Introduction to Quadratic Forms}, Springer-Verlag (1973).
		
		\bibitem[Pe]{Pe} M. Peters, \textit{Summen von Quadraten in Zahlringen}, J. Reine Angew. Math. 268/269, 318--323 (1974).
		
		\bibitem[Pf]{Pf} A. Pfister, \textit{Quadratic forms with applications to algebraic geometry and topology}, London Math. Soc. Lect. Notes 217, Cambridge University Press (1995).
		
		\bibitem[Sa1]{Sa00} H. Sasaki, \textit{Sums of squares of integral linear forms}, J. Austral. Math. Soc. Ser. A 69 298--302 (2000).
		
		\bibitem[Sa2]{SaJapan} H. Sasaki, \textit{Sums of squares of totally positive definite quadratic forms over real quadratic field $\Q(\sqrt{5})$} (in Japanese), Otemae Junior College Research bulletin 25, 407--412 (2005).
		
		\bibitem[Sch]{Sch} R. Scharlau, \textit{On the Pythagoras number of orders in totally real number fields}, J. Reine Angew. Math. 316, 208--210 (1980).
		
		\bibitem[Sch2]{Sch2} R. Scharlau, \textit{Zur Darstellbarkeit von totalreellen ganzen algebraischen Zahlen durch Summen von Quadraten}, dissertation at Universit{\"a}t Bielefeld (1979).
		
		\bibitem[Ti]{Ti} M. Tinkov\'a, \textit{On the Pythagoras number of the simplest cubic fields}, \href{https://arxiv.org/abs/2101.11384}{arXiv:2101.11384}.
	\end{thebibliography}
	
\end{document}